\documentclass[a4,12pt,reqno]{amsart}
\usepackage{amsmath}
\usepackage[english]{babel} 
\usepackage{amsthm}
\usepackage{mathalfa}
\usepackage{frontespizio}
\usepackage{amsfonts}
\usepackage{amssymb}
\usepackage{lmodern}
\usepackage[T1]{fontenc}
\usepackage[utf8]{inputenc} 
\usepackage{latexsym} 
\usepackage{graphics} 
\usepackage{euscript} 
\usepackage{mathrsfs} 
\usepackage{verbatim}
\usepackage[dvips]{curves}
\usepackage{tikz-cd}
\usepackage{graphicx}
\usepackage{subfigure}
\usepackage{color}
\usepackage[a4paper,top=2.5cm,bottom=2.5cm,left=2cm,right=2cm]{geometry}
\usepackage{hyperref,xcolor}
\hypersetup{
	pdfborder={0 0 0},
	colorlinks,
	linkcolor=red,
	citecolor=blue
}

\newcommand{ \Rn} {\mathbb{R}^n}
\newcommand{ \Rm} {\mathbb{R}^m}
\newcommand{ \Ru} {\mathbb{R}}

\newcommand{\Om}{\Omega}

\newcommand{\mA}{\mathcal{A}}

\newcommand{\scu}{\longrightarrow}

\newcommand{\lpl}[1]{L^{p}_{loc}(#1)}

\newcommand{\anso}[1]{W^{1,p}_X(#1)}
\newcommand{\ansol}[1]{W^{1,p}_{X,loc}(#1)}
\newcommand{\cl}[1]{\overline{#1}}

\newcommand{\lul}{L^1_{loc}(U)}
\newcommand{\ch}{\overline{co}}
\newcommand{\sub}{\partial_{X,N}}
\newcommand{\winf}[1]{W^{1,\infty}(#1)}
\newcommand{\winfx}[1]{W_X^{1,\infty}(#1)}

\newtheorem{thm}{Theorem}[section]
\newtheorem{prop}[thm]{Proposition}

\newtheorem{lem}[thm]{Lemma}

\theoremstyle{definition}

\newcommand{\average}{{\mathchoice {\kern1ex\vcenter{\hrule height.4pt
				width 6pt
				depth0pt} \kern-9.7pt} {\kern1ex\vcenter{\hrule height.4pt width 4.3pt
				depth0pt}
			\kern-7pt} {} {} }}
\newcommand{\ave}{\average\int}

\DeclareMathOperator{\Lip}{Lip}
\DeclareMathOperator{\diver}{div}
\DeclareMathOperator{\diverx}{div_X}   
\DeclareMathOperator{\spann}{span} 
\DeclareMathOperator{\lie}{Lie} 
\title[The Aronsson Equation for Absolute Minimizers of Supremal Functionals]{The Aronsson Equation for Absolute Minimizers of Supremal Functionals in Carnot-Carathéodory Spaces}
\author{Andrea Pinamonti}
\address{Department of Mathematics, University of Trento, Via Sommarive 14, 38123 Povo (Trento), Italy}
\email{andrea.pinamonti@unitn.it}
\author{Simone Verzellesi}
\address{Department of Mathematics, University of Trento, Via Sommarive 14, 38123 Povo (Trento), Italy}
\email{simone.verzellesi@unitn.it}
\author{Changyou Wang}
\address{Department of Mathematics, Purdue University 150 N. University Street
	West Lafayette, IN 47907-2067}
\email{wang2482@purdue.edu}

\numberwithin{equation}{section}

\begin{document}
	\maketitle
	\begin{abstract}
		Given a $C^2$ family of vector fields $X_1,\ldots,X_m$ which induces a continuous Carnot-Carath\'eodory distance, we show that any absolute minimizer of a supremal functional defined by a $C^2$ quasiconvex Hamiltonian $f(x,z,p)$, allowing $z$-variable dependence, is a viscosity solution to the Aronsson equation
		\begin{equation*}
			-\langle X(f(x,u(x),Xu(x))), D_pf(x,u(x),Xu(x))\rangle=0.
		\end{equation*}
		
	\end{abstract}
	\section{Introduction}
	The study of variational problems in $L^\infty$ is very often a good starting point to set up problems coming both from theoretic issues and from real applications. The earliest works in this direction are due to Aronsson (\cite{aronsson1, aronsson2}). In these seminal papers, the author studied the connection between Lipschitz extension problems and PDEs, introducing the notion of \emph{absolute minimizing Lipschitz extension (AMLE)} and showing that a $C^2$ function is an AMLE if and only if it satisfies the \emph{infinity Laplace equation}
	\begin{equation}\label{infinity}
		-\sum_{i,j=1}^n\frac{\partial u}{\partial x_i}\frac{\partial u}{\partial x_j}\frac{\partial^2 u}{\partial x_i\partial x_j}=0.
	\end{equation}
	Anyway, as Aronsson observed (\cite{aronsson3}), there are examples of AMLE which are not $C^2$, and thus solving equation \eqref{infinity} only in a formal sense. The problem was solved by Jensen. In the celebrated paper \cite{jensen1}, the author exploited the machinery of viscosity solutions introduced by Crandall and Lions in \cite{crandalllions} (cf. also \cite{visco} for an exaustive account on the topic), and showed that being an AMLE is equivalent to being a viscosity solution to \eqref{infinity}. Moreover, he showed that viscosity solutions to problem \eqref{infinity} are unique, provided a Dirichlet boundary datum is assigned.\\
	One step further was made by Barron, Jensen and Wang (\cite{bjw}), who started the study of $L^\infty$ variational functionals $F$ which are usually known as \emph{supremal functionals}, that is
	\begin{equation*}
		F(u,V):=\|f(x,u(x),Du(x))\|_{L^\infty(V)}\qquad u\in\winf{U}, V\in\mA.
	\end{equation*}
	where throughout the paper $U$ is an open and connected subset of $\Rn$, $\mA$ is the class of all open subsets of $U$ and $f$ is a suitable continuous non-negative function. In particular, they generalized the notion of AMLE to the one of \emph{absolute minimizer} of the functional $F$, that is a function $u\in\winf{U}$ such that
	\begin{equation*}
		F(u,V)\leq F(v,V)
	\end{equation*}
	for any $V\Subset U$ and for any $v\in\winf{V}$ with $v|_{\partial V}=u|_{\partial V}$.
	The authors of \cite{bjw} showed that any absolute minimizer of $F$ is a solution, in the viscosity sense, of the so-called \emph{Aronsson equation}
	\begin{equation*}
		-\sum_{i=1}^n\frac{\partial}{\partial x_i}(f(x,u(x),Du(x)))\frac{\partial f}{\partial p_i} (x,u(x),Du(x))=0,
	\end{equation*}
	provided that, among the other things, $f$ is $C^2$ and $p\mapsto f(x,s,p)$ is \emph{strictly quasiconvex}, where
	we call a function $g:\Rn\scu\Ru$ (strictly) quasiconvex whenever 
	\begin{equation*}
		g(tp_1+(1-t)p_2)\leq\,(<)\,\max\{g(p_1),g(p_2)\}
	\end{equation*}
	for any $p_1,p_2\in\Rm$ with $p_1\not=p_2$ and $t\in(0,1)$.
	This result generalizes the previous ones, in the sense that, in the particular case in which $f(p)=|p|^2$, the notion of absolute minimizer reduces to the one of AMLE and the Aronsson equation becomes the infinity Laplace equation. Many improvements of the results in \cite{bjw} have been achieved by Crandall (\cite{crandall}), both weakening some assumptions and exploiting a concise and elegant proof, and by Crandall, Wang and Yu (\cite{cyw}), dealing with the more natural assumption of $C^1$ Hamiltonians.\\
	More recently, Bieske and Capogna (\cite{B,BC}) studied the derivation of the Aronsson equation, and the question of uniqueness of absolute minimizers, in the setting of Carnot groups and for the case $f(p)=|p|^2$. Later, Wang (\cite{wang}) moved the focus on the possibility to extend the previous results to more general frameworks, and started the study of supremal functionals defined in the setting of \emph{Carnot-Carathéodory spaces}. We stress that this point of view is pretty general and encompasses, among other things, the Euclidean setting and many interesting sub-Riemannian manifolds. On the other hand its rich analytical structure allows to study many interesting problems in great generality (see for example \cite{EPV,MSC,MPSC,MPSC2} and references therein).

	In order to better introduce this issue we recall some terminology and some well known facts.\\
	Given a family $X=(X_1\ldots,X_m)$ of locally Lipschitz vector fields defined on $U$, we say that an absolutely continuous curve $\gamma:[0,\delta]\scu U$ is \emph{horizontal} when there are measurable functions $a_1(t),\ldots,a_m(t)$ with
	\begin{equation*}
		\Dot{\gamma}(t)=\sum_{j=1}^ma_j(t)X_j(\gamma(t))\qquad\text{ for a.e. }t\in[0,\delta],
	\end{equation*}
	and we say that it is \emph{subunit} whenever it is horizontal with $\sum_{j=1}^m a_j^2(t)\leq 1$ for a.e. $t\in[0,\delta]$.
	Moreover, we define the \emph{Carnot-Carathéodory distance} on $U$ as
	\begin{equation*}
		d_X(x,y):=\inf\left\{\int_0^1|\Dot{\gamma}(t)|dt\,:\,\gamma:[0,1]\scu U\text{ is subunit, $\gamma(0)=x$ and $\gamma(1)=y$}\right\}.
	\end{equation*}
	If $d_X$ is a (finite) distance on $U$, we say that $(U,d_X)$ is a \emph{Carnot-Carathéodory space}.
	Moreover, we denote by $C(x)$ the $m\times n$ matrix defined as \[C(x):=[c_{j,i}(x)]_{\substack{{i=1,\dots,n}\\{j=1,\dots,
				m}}},\]
	where for each $j=1\ldots,m$ we have $X_j:=\sum_{i=1}^n c_{j,i}\frac{\partial}{\partial x_i}$.
	If $u\in\lul$, we define the distributional $X$-gradient (or \emph{horizontal gradient}) of $u$ as 
	\begin{equation*}
		\langle Xu,\varphi\rangle:=-\int_U u \diver (\varphi\cdot C(x))dx\qquad\text{ for any }\varphi\in C^\infty_c(U,\Rm).
	\end{equation*}
	Finally, if $p\in[1,+\infty]$, we define the \emph{horizontal Sobolev spaces} as
	$$\anso{U}:=\{u\in L^p(U)\,:\,Xu\in L^p(U,\Rm)\}$$
	and
	$$\ansol{U}:=\{u\in\lpl{U}\,:\,u|_{V}\in\anso{V},\quad\forall\, V\Subset U\}.$$
	In \cite{wang} the author adapted in the obvious way the notion of absolute minimizer to this framework, and showed, under mild assumptions on the generating family of vector fields, that any absolute minimizer of the supremal functional defined by
	\begin{equation*}
		F(u,V):=\|f(x,Xu(x))\|_{L^\infty(V)}
	\end{equation*}
	is a viscosity solution of the equation
	\begin{equation*}
		-\sum_{i=1}^m X_i(f(x,Xu(x)))\frac{\partial f}{\partial p_i} (x,Xu(x))=0,
	\end{equation*}
	provided that $p\mapsto f(x,p)$ is quasiconvex, that $f$ is homogeneous of degree $\alpha\geq1$ and that $D_p f(0,0)=0$. Finally, Wang and Yu (\cite{wy}) improved the previous result by requiring only $C^1$ regularity for $f$ and dropping the assumption that $D_pf(0,0)=0$ (see also \cite{dragoni} for some more specific results for the case $f(p)=|p|^2$). However, neither \cite{wang} nor \cite{wy} studied the problem for Hamiltonian functions $f$ that allow
	$z$-variable dependence. \\
	
	In the present paper we generalize the results in \cite{crandall} and \cite{wang}, showing that any absolute minimizer of the functional
	\begin{equation*}
		F(u,V):=\|f(x,u(x),Xu(x))\|_{L^\infty(V)}
	\end{equation*}
	is a viscosity solution to the Aronsson equation
	\begin{equation*}
		-\sum_{i=1}^m X_i(f(x,u(x),Xu(x)))\frac{\partial f}{\partial p_i} (x,u(x),Xu(x))=0,
	\end{equation*}
	provided that the following conditions hold.
	\begin{itemize}
		\item [$(X1)$] $d_X$ is a distance on $U$, and it is continuous with respect to the Euclidean topology.
		\item[$(X2)$] $X_i\in C^2(U,\Rn)$ for any $i=1,\ldots,m.$
		\item[$(f1)$] $f\in C^2(\Om\times\Ru\times\Rm,[0,\infty))$.
		\item[$(f2)$] $p\mapsto f(x,s,p)$ is quasiconvex for any $x\in\Om$ and for any $s\in\Ru$.
	\end{itemize}
	The strategy of our proof, strongly inspired by \cite{crandall}, is divided into five steps.
	\begin{itemize}
		\item [\textbf{Step 1.}] Arguing by contradiction, we assume that there is an absolute minimizer which fails to be a viscosity subsolution to the Aronsson equation. Therefore, without loss of generality, we assume that there exists a function $\phi\in C^2(U)$, which touches $u$ from above in $0$, such that 
		\begin{equation*}
			-\sum_{i=1}^m X_i(f(0,\phi(0),X\phi(0)))\frac{\partial f}{\partial p_i} (0,\phi(0),X\phi(0))>0.
		\end{equation*}
		\item [\textbf{Step 2.}] Exploiting ideas from \cite{crandall,wang}, we build a family $(\Psi_\varepsilon)_\varepsilon$ of classical solutions to the Hamilton-Jacobi equation
		$$f(x,\Psi_\varepsilon(x),X\Psi_\varepsilon(x))=f(0,\phi(0)-\varepsilon,X\phi(0))$$
		for which we have some continuity properties.
		\item [\textbf{Step 3.}] We find and open set $\mathcal{N}_\varepsilon$ which allows to consider $\Psi_\varepsilon$ as a competitor in the definition of absolute minimizer.
		\item [\textbf{Step 4.}] By an appropriate change of variables we reduce to the case in which $s\mapsto f(x,s,p)$
		is non-decreasing in a neighborhood of $(0,\phi(0),X\phi(0))$.
		\item[\textbf{Step 5.}] We show the solvability of a suitable system of ODEs to get a family of $C^1$ curves $(\gamma_\varepsilon)_\varepsilon$, and we show that there is a choice among such curves which allows to reach a contradiction.
	\end{itemize}
	In particular, the last step involves some preliminary results about differentiability in Carnot-Carathéodory spaces which we tackled, inspired again by \cite{crandall}, by suitably adapting the notion of \emph{subdifferential} introduced in \cite{clarke}.\\
	From one hand, our result generalizes \cite{crandall} to the more general setting of Carnot-Carathéodory spaces. Moreover, differently from \cite{wang}, we allow also the function dependence of the hamiltonian and we drop the requirement $D_pf(0,0)=0$. Finally, the results in \cite{wy}, apart from not allowing the function dependence of the hamiltonian, are achieved under the H\"ormander condition, which is known to be stronger than $(X1)$. On the other hand our techniques strongly relies on the $C^2$ regularity of the hamiltonian, which is on the contrary weakened in \cite{wy}.\\
	The paper is organized as follows. In Section 2 we recall some preliminaries about Carnot Carathéodory spaces, viscosity solutions, absolute minimizers and quasiconvex functions, we introduce the aforementioned notion of subdifferential and we shows some useful properties of differentiability along horizontal curves. In Section 3 we state and prove the main result of this paper.

	\section{Preliminaries}
	\subsection{Notation}
	Unless otherwise specified, we let $m,n\in\mathbb{N}\setminus\{0\}$ with $m\leq n$, we denote by $U$ an open and connected subset of $\Rn$ and by $\mA$ the class of all open subsets of $U$. Given two open sets $A$ and $B$, we write $A\Subset B$ whenever $\overline{A}\subseteq B$. If $E\subseteq \Rn$, we set $\ch E$ to be the closure of 
	\begin{equation*}
		co E:=\bigcap\{C\,:\,C\text{ is convex and }E\subseteq C\}.
	\end{equation*}
	It is easy to see that $co E$ is convex and that $\ch E$ is closed and convex. Moreover we set $\Lambda_n:=\{(\lambda_1,\ldots,\lambda_n)\,:\,0\leq\lambda_j\leq 1,\,\sum_{j=1}^n\lambda_j=1\}$.
	For any $u,v\in\Rn$, we denote by $\langle u,v\rangle $ the Euclidean scalar product, and by $|v|$ the induced norm. We let $S^m$ be the class of all $m\times m$ symmetric matrices with real coefficients. Moreover, if $A$ is a $p\times q$ matrix and $B$ is a $q\times r$ matrix, we let $A\cdot B$ be the usual matrix product. We denote by $\mathcal{L}^n$ the restriction to $U$ of the $n$-th dimensional Lebesgue measure, and for any set $E\subseteq U$ we write $|E|:=\mathcal{L}^n(E)$.
	Given $x\in\Rn$ and $R>0$ we let $B_R(x):=\{y\in\Rn\,:\,|x-y|<R\}$.
	If we have a function $g\in\lul$ and $x\in U$ is a Lebesgue point of $g$, when we write $g(x)$ we always mean that $$g(x)=\lim_{r\to 0^+}\ave_{B_r(0)}g(y)dy.$$
	If $f(x,s,p)$ is a regular function defined on $U\times\Ru\times\Rm$, we denote by $D_xf=(D_{x_1}f,\ldots,D_{x_n}f)$, $D_sf$ and $D_pf=(D_{p_1}f,\ldots,D_{p_m}f)$ the partial gradients of $f$ with respect to the variables $x,s$ and $p$ respectively. In general we mean gradients as row vectors.
	\subsection{Carnot-Carathéodory spaces}
	Assume that we have a family $X_1,\ldots,X_m$ of locally Lipschitz vector fields defined on $U$.
	Given $k\geq 1$, we define $C^k_{X}(U)$ by
	\begin{equation*}
		C^k_X(U):=\{u\in C(U)\,:\,\exists X_{i_1}\cdots X_{i_s}u\in C(U)\text{ for any $(i_1,\ldots,i_s)\in\{1,\ldots,m\}^s$ and $1\leq s\leq k$} \}.
	\end{equation*}
	Therefore, whenever we have a function $u\in C^2_X(U)$, we can define its \emph{horizontal Hessian} $X^2u\in C(U,S^m)$ as
	\begin{equation*}
		X^2u(x)_{ij}:=\frac{X_iX_ju(x)+X_jX_iu(x)}{2}
	\end{equation*}
	for any $x\in U$ and $i,j=1,\ldots,m$.
	When in addition $(U,d_X)$ is a Carnot-Carathéodory space, we can define the \emph{Horizontal Lipschitz space} as
	\begin{equation*}
		\Lip_X(U):=\left\{u:U\scu\overline{\Ru}\,:\,\sup_{x\neq y}\frac{u(x)-u(y)}{d_X(x,y)}<+\infty\right\}.
	\end{equation*}
	It is well known, and we refer to \cite{franchiserapserra}, that the equality
	\begin{equation}\label{wl}
		\winfx{U}=\Lip_X(U)
	\end{equation}
	holds.
	In this paper, unless otherwise specified, we always assume that 
	\begin{itemize}
		\item [$(X1)$] $d_X$ is a distance on $U$, and it is continuous with respect to the Euclidean topology.
	\end{itemize}
	In particular, we point out that, if $(X1)$ holds, then each function $u\in W^{1,\infty}_{X,loc}(U)$ admits a continuous representative, that is
	\begin{equation}\label{continuita}
		W^{1,\infty}_{X,loc}(U)\subseteq C(U).
	\end{equation}
	Indeed, if $u\in W^{1,\infty}_{X,loc}(U)$ and $x,y\in U$, then, if $x,y\in K\Subset\Om$ and thanks to \eqref{wl}, it holds that
	\begin{equation*}
		|u(x)-u(y)|=d_X(x,y)\frac{|u(x)-u(y)|}{d_X(x,y)}\leq d_X(x,y)\sup_{z\neq w\in K}\frac{u(z)-u(w)}{d_X(z,w)},
	\end{equation*}
	and the right side goes to zero as $x\to y$ in virtue of $(X1)$.
	Therefore, in the following we identify $u\in W^{1,\infty}_{X,loc}(U)$ with its continuous representative.\\
	As it is well known, assumption $(X1)$ is quit mild in this framework, since it includes many relevant situations. 
	Just to mention the most famous instance, we recall that a family $X_1,\ldots,X_m$ satisfies the \emph{H\"ormander condition} whenever each $X_j$ is a smooth vector fields and it holds that
	$$\spann\{\lie(X_1(x),\ldots,X_m(x))\}=\Rn\qquad \text{ for any }x\in U,$$
	where $\lie(X_1(x),\ldots,X_m(x))$ denotes the Lie algenra generated by $X_1(x),\ldots,X_m(x)$.
	From \cite{nagel,gromov} we know the following result.
	\begin{prop}
		Assume that $X$ satisfies the H\"ormander condition. Then the following properties hold.
		\begin{itemize}
			\item[(i)] $(U,d_X)$ is a Carnot-Carathéodory space.
			\item[(i)] For any compact set $K\subseteq U$ there exists a positive constant $C_K$ such that 
			\begin{equation*}
				C_K^{-1}|x-y|\leq d_X(x,y)\leq C_K|x-y|^{\frac{1}{r}}\qquad\text{ for any }x,y\in K,
			\end{equation*}
			being $r$ the nilpotency step of $\lie(X_1,\ldots,X_m)$.
		\end{itemize}
	\end{prop}
	Hence H\"ormander vector fields are examples of vector fields satisfying $(X1)$.

	\subsection{Subgradient in Carnot-Carathéodory spaces}
	When $u\in W_{X,\emph{loc}}^{1,\infty}(U)$ and $N\subseteq U$ is any Lebesgue-negligible set which contains all the non-Lebesgue points of $Xu$, we define the \emph{$(X,N)$-subgradient} of $u$ as
	\begin{equation*}
		\sub u(x):=\ch\{\lim_{n\to\infty}Xu(y_n)\,:\,y_n\to x,\,y_n\notin N\text{ and }\exists\lim_{n\to\infty}Xu(y_n)\}
	\end{equation*}
	for any $x\in U$. This notion is inspired by the classical subdifferential introduced in \cite{clarke}. Anyway, since our hypoteses are too general to ensure the validity of a Rademacher-type Theorem for functions in $\Lip_X(U)$ (cf. \cite{MSC}), these two objects enjoys different properties. Therefore the Euclidean $(X,N)$-subgradient, i.e. when $X=(\partial_1,\ldots,\partial_n)$, does not coincide in general with Clarke's subdifferential.
	
	\begin{prop}\label{equivalenza}
		Let $u$ and $N$ be as above. Then the following facts hold.
		\begin{itemize}
			\item[$(i)$] $\sub u(x)$ is a non-empty, convex, closed and bounded subset of $\Rm$ for any $x\in U$;
			\item [$(ii)$] for any $x\in U$
			\begin{equation*}
				\sub u(x)=\bigcap_{k=1}^\infty\ch \{Xu(y)\,:\,y\in B_{1/k}(x)\setminus N\};
			\end{equation*}
			\item[$(iii)$] if $u\in C^1_X(U)$, then 
			\begin{equation*}
				\sub u(x)=\{Xu(x)\}
			\end{equation*}
			for any $x\in U$.
		\end{itemize}
	\end{prop}
	\begin{proof}
		We start by proving $(i)$. We fix $x\in U$ and show that $\sub u(x)\neq\emptyset$. Let $r>0$ be small enough to have $B_r(x)\Subset U$. Then $u\in\winfx{B_r(x)}$. So we set $L:=\|Xu\|_{L^\infty(B_r(x))}$. Let $(r_n)_{n}\subseteq (0,r)$ with $r_n\searrow 0$. Then, for any $n\in\mathbb{N}$, take $y_n\in B_{r_n}(x)\setminus N$. Then clearly $y_n$ tends to $x$. Moreover, being $y_n$ a Lebesgue point of $Xu$, it follows that 
		$$|Xu(y_n)|=\left|\lim_{s\to 0^+}\ave_{B_s(y_n)}Xu(z)dz\right|\leq\lim_{s\to 0^+}\ave_{B_s(y_n)}|Xu(z)|dz\leq L,$$
		and so $(Xu(y_n))_n$ is bounded in $\Rm$. Therefore, up to a subsequence, we can assume that its limit exsts, that is $\sub u(x)$ is non-empty. From the above proof it is easy to see that $\sub u(x)$ is bounded, while convexity and closure follows directly from its definition.
		Let us prove $(ii)$. We fix $x\in U$ and start by proving the left-to-right inclusion. As the right set is convex and closed, it is sufficient to show that any $z$ of the form
		\begin{equation*}
			z=\lim_{n\to\infty}Xu(y_n),
		\end{equation*}
		with $y_n\to x$ and $y_n\notin N$, belongs to 
		$$\ch \{Xu(y)\,:\,y\in B_{1/k}(x)\setminus N\}$$
		for any $k\in\mathbb{N}\setminus\{0\}$. As $y_n$ tends to $x$ we get that $y_n\in B_{1/k}(x)\setminus N$ for $n$ sufficiently large. Therefore, as the conclusion follows for each $X(y_n)$ and the right set is closed, we have proved the desired inclusion. The proof of the converse inclusion follows form the two following Lemmas, which will be proved at the end of this paper to avoid confusion.
		\begin{lem}\label{lemmauno}
			Let 
			$$S:=\left\{\lim_{n\to\infty}Xu(y_n)\,:\,y_n\to x,\,y_n\notin N\text{ and }\exists\lim_{n\to\infty}Xu(y_n)\right\}$$
			and, for any $k\geq 1$, let
			$$A_k=\{Xu(y)\,:\,y\in B_{1/k}(x)\setminus N\}.$$
			Then it follows that
			\begin{equation*}
				\bigcap_{k=1}^\infty\cl A_k \subseteq S.
			\end{equation*}
		\end{lem}
		\begin{lem}\label{lemmadue}
			Let $(A_k)_k$ be a decreasing sequence of non-empty bounded subsets of $\Rm$, and let $S$ be a non-empty, bounded subset of $\Rm$. Assume that 
			\begin{equation*}
				\bigcap_{k=1}^\infty\cl A_k \subseteq S.
			\end{equation*}
			Then it follows that
			\begin{equation*}
				\bigcap_{k=1}^\infty\ch (\cl A_k) \subseteq \ch(S).
			\end{equation*}
		\end{lem}
		
		Now we prove $(iii)$. Let $x\in U$ and let $(y_n)_n\subseteq U\setminus N$ converges to $x$. Then from the continuity of $Xu$ it follows that $\lim_{n\to\infty}Xu(u_n)=Xu(x)$. Since $\{Xu(x)\}$ is convex and closed, this implies that $\sub u(x)\subseteq \{Xu(x)\}$. Conversely, being $N$ negligible, there exists a sequence $(y_n)_n\subseteq U\setminus N$ which converges to $x$. Again thanks to the continuity of $Xu$, the converse inclusion follows.
	\end{proof}
	With the following proposition we see that the notion of $(N,X)$-subgradient, in analogy with the Euclidean setting, is the right tool to deal with differentiability of $X$-Lipschitz functions along horizontal curves.
	\begin{prop}\label{derivabilitaac}
		Assume that $X$ satisfies $(X1)$.
		Let $1\leq p\leq+\infty$, let $u\in W^{1,\infty}_{X,loc}(U)$ and let $\gamma\in \emph{AC}([-\beta,\beta],U)$ be a horizontal curve with
		\begin{equation*}
			\Dot{\gamma}(t)=C(\gamma(t))^T\cdot A(t)
		\end{equation*}
		and $A\in L^p((-\beta,\beta),\Rm)$. Then the curve $t\mapsto u(\gamma(t))$ belongs to $W^{1,p}(-\beta,\beta)$, and there exists a function $g\in L^\infty((-\beta,\beta),\Rm)$ such that 
		\begin{equation*}
			\frac{d u(\gamma(t))}{d t}=g(t)\cdot A(t)
		\end{equation*}
		for a.e. $t\in (-\beta,\beta)$. Moreover
		\begin{equation*}
			g(t)\in\sub u(\gamma (t))
		\end{equation*}
		for a.e. $t\in (-\beta,\beta).$
	\end{prop}
	\begin{proof}
		Let $(\varrho_\delta)_\delta$ be a sequence of spherically symmetric mollifiers, and let $N$ be any negligible set which contains all the non-Lebesgue points of $Xu$.
		If $\delta$ is sufficiently small and we define $u_\delta$ and $(Xu)_\delta$ to be the standard convolutions, we have that these functions are smooth on a bounded open set, say $V$, such that $V\Subset U$ and $V$ contains the support of $\gamma$. Moreover, as $X$ satisfies $(X1)$, from \cite{wang} we know that there exists a non-negative and non-decreasing function $w(\delta)$ (depending on the chosen function $u$) defined in a right neighborhood of $0$, such that 
		\begin{equation*}
			\lim_{\delta\to 0^+}w(\delta)=0
		\end{equation*}
		and moreover
		\begin{equation}\label{stimadelta}
			|X(u_\delta)(x)-(Xu)_\delta(x)|\leq w(\delta)
		\end{equation}
		for any $x\in V$.
		As $u_\delta$ is $C^1$ and $\gamma$ is absolutely continuous, from standard calculus we have that
		\begin{equation}\label{tfc}
			\begin{split}
				u_\delta(\gamma(t))-u_\delta(\gamma(0))&=\int_0^tD(u_\delta)(\gamma(s))\cdot\Dot{\gamma}(s)ds\\
				&=\int_0^tD(u_\delta)(\gamma(s))\cdot C(\gamma(s))^T\cdot A(s)ds\\
				&=\int_0^tX(u_\delta)(\gamma(s))\cdot A(s)ds.
			\end{split}
		\end{equation}
		Let us consider now the sequence of functions $X(u_{1/n})(\gamma(\cdot))$. It is easy to see that it is bounded in $L^\infty((-\beta,\beta),\Rm)$. Therefore (up to a subsequence) there exists a function $g\in L^\infty((-\beta,\beta),\Rm)$ such that
		\begin{equation}\label{weakconv}
			X(u_{1/n})(\gamma(\cdot))\rightharpoonup^* g(\cdot)\qquad \text{in }L^\infty((-\beta,\beta),\Rm)
		\end{equation}
		as $n$ goes to infinity, and so in particular
		\begin{equation}\label{weakconv2}
			X(u_{1/n})(\gamma(\cdot))\rightharpoonup g(\cdot)\qquad \text{in }L^2((-\beta,\beta),\Rm)
		\end{equation}
		as $n$ goes to infinity. Since $u$ is continuous, then by well known results we have that $u_\delta$ converges uniformly to $u$ on $V$. Therefore, passing to the limit in \eqref{tfc}, noticing in particular that $A\in L^1((-\beta,\beta),\Rm)$ and exploiting \eqref{weakconv}, we obtain that 
		\begin{equation*}
			u(\gamma(t))-u(\gamma(0))=\int_0^tg(s)\cdot A(s)ds.
		\end{equation*}
		We are left to show that $g(t)\in\sub u(\gamma(t))$ for a.e. $t\in(-\beta,\beta)$.
		Let us notice that, since for any $x\in V$ we have that
		\begin{equation*}
			(Xu)_\delta (x)=\int_{B_\delta(x)\setminus N}\varrho_\delta(y-x)Xu(y)dy,
		\end{equation*}
		it follows that
		\begin{equation}\label{ch1}
			(Xu)_\delta(x)\in\ch \{Xu(y)\,:\,y\in B_\delta(x)\setminus N\}
		\end{equation}
		for any $x\in V$. Indeed, recalling that $Xu\in L^\infty(B_\delta(x)\setminus N)$ for $\delta$ small enough, setting $m:=\inf_{B_\delta(x)\setminus N}Xu$ and $M:=\sup_{B_\delta (x)\setminus N}Xu$, it holds that
		\begin{equation*}
			m=m\int_{B_\delta (x)\setminus N}\rho_\delta(x-y)dy\leq (Xu)_\delta (x)\leq M\int_{B_\delta (x)\setminus N}\rho_\delta(x-y)dy=M,
		\end{equation*}
		and so $(Xu)_\delta(x)\in [m,M]$. Therefore, noticing that
		\begin{equation*}
			[m,M]=\ch \{m,M\}\subseteq\ch\{Xu(y)\,:\,y\in B_\delta(x)\setminus N\}\subseteq[m,M],
		\end{equation*}
		then \eqref{ch1} follows. Thanks to \eqref{weakconv2} and Mazur's Lemma (cf. e.g. \cite[Corollary 3.9]{brezis}), for each $m\in\mathbb{N}$ there are convex combinations of $X(u_{1/n})(\gamma(\cdot))$ converging strongly to $g$ in $L^2((-\beta,\beta),\Rm)$, that is
		\begin{equation*}
			v_m(\cdot):=\sum_{n=M_m}^{N_m}a_{m,n}X(u_{1/n})(\gamma(\cdot))\longrightarrow g(\cdot) \qquad \text{in }L^2((-\beta,\beta),\Rm),
		\end{equation*}
		with $M_m<N_m$ and $\lim_{m\to\infty}M_m=+\infty$. Moreover (again up to a subsequence) we can assume that the above convergence holds pointwise for a.e. $t\in(-\beta,\beta)$.
		Let us define now 
		\begin{equation*}
			z_m(\cdot):=\sum_{n=M_m}^{N_m}a_{m,n}(Xu)_{1/n}(\gamma(\cdot)).
		\end{equation*}
		Then, thanks to \eqref{stimadelta} we have that
		\begin{equation*}
			\begin{split}
				|z_m(t)-g(t)|&\leq\sum_{n=M_m}^{N_m}a_{m,n}|X(u_{1/n})(\gamma(t))-(Xu)_{1/n}(\gamma(t))|+|v_m(t)-g(t)|\\
				&\leq\sum_{n=M_m}^{N_m}a_{m,n}w(1/n)+|v_m(t)-g(t)|\\
				&\leq\sum_{n=M_m}^{N_m}a_{m,n}w(1/M_m)+|v_m(t)-g(t)|\\
				&=w(1/M_m)+|v_m(t)-g(t)|,
			\end{split}
		\end{equation*}
		which implies that $z_m$ converges to $g$ pointwise for a.e. $t\in(-\beta,\beta)$ as $m\to\infty$. Moreover, thanks to \eqref{ch1} and the definition of $z_m$ it follows easily that
		\begin{equation*}
			z_m(t)\in\ch \{Xu(y)\,:\,y\in B_{1/M_m}(\gamma(t))\setminus N\}\subseteq \ch \{Xu(y)\,:\,y\in B_{1/k}(\gamma(t))\setminus N\}
		\end{equation*}
		for any $t\in(-\beta,\beta)$ and for any $k\leq M_m$. Therefore, thanks to the pointwise convergence as $m\to\infty$, we get that
		\begin{equation*}
			g(t)\in\bigcap_{k=1}^\infty\ch \{Xu(y)\,:\,y\in B_{1/k}(\gamma(t))\setminus N\}.
		\end{equation*}
		for a.e. $t\in(-\beta,\beta)$. Finally, thanks to Proposition \ref{equivalenza}, the thesis follows.
	\end{proof}
	As a corollary of the previous proposition we have the following result.
	\begin{prop}\label{derivabilitaacc1}
		Assume that $X$ satisfies $(X1)$.
		Let $u\in C^{1}_{X}(U)$ and let $\gamma\in C^1([-\beta,\beta],U)$ be a horizontal curve with
		\begin{equation*}
			\Dot{\gamma}(t)=C(\gamma(t))^T\cdot A(t)
		\end{equation*}
		and $A\in C([-\beta,\beta],\Rm)$. Then the curve $t\mapsto u(\gamma(t))$ belongs to $C^1(-\beta,\beta)$ and 
		\begin{equation*}
			\frac{d u(\gamma(t))}{d t}=Xu(\gamma(t))\cdot A(t)
		\end{equation*}
		for any $t\in (-\beta,\beta)$.
	\end{prop}
	We conclude this section with a useful property which links subgradients and quasiconvex functions.
	\begin{lem}\label{sublevel}
		Let $f\in C(U\times\Ru\times\Rm)$ be a non-negative function which satisfies $(f2)$. Let $u\in W^{1,\infty}_{X,\emph{loc}}(U)$, $V\in\mA$ and $K\geq0$ such that
		\begin{equation*}
			f(x,u(x),Xu(x))\leq K
		\end{equation*}
		for a.e. $x\in V$. Let $N$ be a Lebesgue-negligible subset of $V$ containing all the points where the previous inequality fails and all the non-Lebesgue points of $Xu$. Then it follows that
		\begin{equation*}
			f(x,u(x),w)\leq K
		\end{equation*}
		for any $x\in V$ and for any $w\in\sub u(x)$.
	\end{lem}
	\begin{proof}
		Let $x\in V$ be fixed and let $w\in\sub u(x)$. Then there exists a sequence 
		$$(w_h)_h\subseteq co\left\{\lim_{n\to\infty}Xu(y_n)\,:\,y_n\to x,\,y_n\notin N\text{ and }\exists\lim_{n\to\infty}Xu(y_n)\right\}$$
		converging to $w$ in $\Rm$. If we are able to prove the claim for each $w_h$, the thesis follows from the continuity of $f$ in the third argument. Fix then $h$. Thanks to Carathéodory Theorem (cf. \cite[Theorem 1.2]{dacorogna}) there are $(\lambda^h_1,\ldots,\lambda^h_{n+1})\in\Lambda_{n+1}$ and $w^h_1,\ldots,w^h_{n+1}$ such that
		$$w^h_j\subseteq \left\{\lim_{n\to\infty}Xu(y_n)\,:\,y_n\to x,\,y_n\notin N\text{ and }\exists\lim_{n\to\infty}Xu(y_n)\right\}$$
		for any $j=1,\ldots,n+1$ and
		$$w^h=\sum_{j=1}^{n+1}\lambda^h_jw^h_j.$$
		Again, if we are able to show the claim for each $w^h_j$, we are done thanks to the convexity of sublevel sets of $f$. Let us fix $j$ and take a sequence $(y_s)_s\subseteq V\setminus N$ converging to $x$ and such that $w^h_j=\lim_{s\to\infty}X(y_s)$.
		As the the map $(x,\eta)\mapsto f(x,u(x),\eta)$ is continuous, and thanks again to the global continuity of $f$, we conclude that
		\begin{equation*}
			f(x,u(x),w^h_j)=\lim_{s\to\infty}f(x,u(x),Xu(y_s))=\lim_{s\to\infty}f(y_s,u(y_s),Xu(y_s))\leq K.
		\end{equation*}
	\end{proof}

	\subsection{Supremal functionals, absolute minimizers and Aronsson equation}
	For sake of completeness we make explicit the definition of supremal functional and of absolute minimizer in the framework of Carnot-Carathéodory spaces.
	Indeed, given a non-negative function $f\in C(U\times\Ru\times\Rm)$, we define its associated \emph{supremal functional} $F:\winfx{U}\times\mA\scu[0,+\infty]$ as
	\begin{equation*}
		F(u,V):=\|f(x,u,Xu)\|_{L^\infty(V)}
	\end{equation*}
	for any $V\in\mA, u\in\winfx{V}$, and we say that $u\in\winfx{U}$ is an \emph{absolute minimizer} of $F$ if
	\begin{equation*}
		F(u,V)\leq F(v,V)
	\end{equation*}
	for any $V\Subset U$ and for any $v\in\winfx{V}$ with $v|_{\partial V}=u|_{\partial V}$.
	Moreover, according to \cite{wang}, we say that a function $A\in C(U\times\Ru\times\Rm\times S^m)$ is \emph{horizontally elliptic} if 
	\begin{equation*}
		A(x,s,p,Z)\leq A(x,s,p,Y)
	\end{equation*}
	whenever $x\in U$, $s\in\Ru$, $p\in\Rm$ and $Z,Y\in S^m$ with $Y\leq Z$.
	If $f$ as above belongs to $C^1(U\times\Ru\times\Rm)$, we can define $A_f:U\times\Ru\times\Rm\times S^m\scu\Ru$ as
	\begin{equation*}
		A_f(x,s,p,Z):=-(Xf(x,s,p)+D_sf(x,s,p)p+D_pf(x,s,p)\cdot Z)\cdot D_pf(x,s,p)^T,
	\end{equation*}
	and we say that 
	\begin{equation}\label{ae}
		A_f[\phi](x):=A_f(x,\phi(x),X\phi(x),X^2\phi(x))=0
	\end{equation}
	is the \emph{Aronsson equation} associated to $F$.
	It is easy to check that $A_f$ is continuous and horizontally elliptic. Moreover, for any $\phi\in C^2(U)$ and $x\in U$ it holds that 
	\begin{equation*}
		A_f[\phi](x)=-X(f(x,\phi,X\phi))\cdot D_pf(x,\phi,X\phi)^T.
	\end{equation*}
	According to \cite{visco,wang} we can now recall the notion of \emph{viscosity solution} to the Aronsson equation. Therefore, we say that a function $u\in C(U)$ is a \emph{viscosity subsolution} to the Aronsson equation if
	\begin{equation*}
		A_f[\phi](x_0)\leq 0
	\end{equation*}
	for any $x_0\in U$ and for any $\phi\in C^2(U)$ such that 
	\begin{equation}\label{maximumvisco}
		0=\phi(x_0)-u(x_0)\leq \phi(x)-u(x)
	\end{equation}
	for any $x$ in a neighbourhood of $x_0$.
	Moreover we say that $u$ is a \emph{viscosity supersolution} if $-u$ is a viscosity subsolution, and finally we say that $u$ is a \emph{viscosity solution} if it is both a subsolution and a supersolution. \\
	We end this section with a straightforward property satisfied by quasiconvex function.
	\begin{prop}\label{quasiconvex}
		Let $g\in C^1(\Rm)$ be a quasiconvex function. Then it holds that
		\begin{equation*}
			g(p)\geq g(q)\implies D_pg(p)\cdot(q-p)\leq0
		\end{equation*}
		for any $p,q\in\Rm$.
	\end{prop}

	\section{The Main Theorem}
	We are ready to state and prove the main theorem of this paper.
	\begin{thm}
		Assume that $(X1),(X2),(f1),(f2)$ hold.
		Then any absolute minimizer of $F$ is a viscosity solution to the Aronsson equation.
	\end{thm}
	\begin{proof} We divide the proof into several steps:\\
	
		\textbf{Step 1.}
		Let $u$ be an absolute minimizer for $F$. It suffices to show that $u$ is a viscosity subsolution to \eqref{ae}, being the other half of the proof completely analogous. Without loss of generality, we assume that $0\in U$. Arguing by contradiction, we assume that $u$ fails to be a subsolution, that is there exists $x_0\in U$, $R_1>0$ and $\phi\in C^2(U)$ such that \eqref{maximumvisco} holds for any $x\in \overline{B_{R_1}(x_0)}$ and 
		\begin{equation}\label{contradiction}
			A_f[\phi](x_0)>0.
		\end{equation}
		Again, without loss of generality we assume that
		$x_0=0$.\\
		
		\textbf{Step 2.}
		We combine ideas form \cite{crandall} and \cite{wang} to achieve the following 
		\begin{lem}\label{lemmacrandall}
			There exist $0<R_2<R_1$, $\epsilon_1>0$, $\mu>0$ and a continuous function $\Psi:[0,\epsilon_1]\times B_{R_2}(0)\scu\Ru$ such that, if we denote $\Psi(\epsilon,x)$ by $\Psi_\epsilon(x)$, it holds that $x\to\Psi_\epsilon(x)\in C^2(B_{R_2}(0))$ for any $\epsilon\in[0,\epsilon_1]$ and 
			\begin{equation}\label{peano1}
				D\Psi_\epsilon \text{ is continuous in } (x,\epsilon)=(0,0).
			\end{equation}
			Moreover, it holds that 
			\begin{equation}\label{phiepsi}
				\begin{gathered}
					\Psi_\epsilon(0)=\phi(0)-\epsilon
					\qquad D\Psi_\epsilon(0)=D\phi(0)
					\qquad D^2\Psi_\epsilon(0)-D^2\phi(0)> 2\mu I_n\\
					f(x,\Psi_\epsilon(x),X\Psi_\epsilon(x))=f(0,\phi(0)-\epsilon,X\phi(0))
				\end{gathered}
			\end{equation}
			for any $x\in B_{R_2}(0)$.
		\end{lem}
		\begin{proof}[Proof of Lemma \ref{lemmacrandall}]
			Lat us define a new function $\overline{f}$ on $U\times\Ru\times\Rn$ as 
			\begin{equation}
				\overline{f}(x,s,\xi):=f(x,s,C(x)\cdot\xi)
			\end{equation}
			for any $x\in U$, $s\in\Ru$ and $\xi\in\Rn$. Then, since $f$ and $X$ are $C^2$, it follows that $\overline{f}\in C^2(U\times\Ru\times\Rn)$.
			Moreover, trivial computations shows that
			\begin{equation}\label{derivata1}
				D_\xi\overline{f}(x,u,\xi)=D_pf(x,u,C(x)\cdot\xi)\cdot C(x),
			\end{equation}
			and that
			\begin{equation}\label{identita1}
				f(x,\varphi(x),X\varphi(x))=\overline{f}(x,\varphi(x),D\varphi(x))
			\end{equation}
			for any $x\in U$ and any $\varphi\in C^2(U)$. Finally, if we let $A_{\overline{f}}\in C(U\times\Ru\times\Rn\times S^n)$ be the Euclidean Aronsson operator associated to $\overline{f}$, i.e. 
			$$A_{\overline{f}}(x,s,\xi,Z):=-(D_x\overline{f}(x,s,\xi)+D_s\overline{f}(x,s,\xi)\xi+D_\xi\overline{f}(x,s,\xi)\cdot Z)\cdot D_\xi\overline{f}(x,s,\xi)^T,$$
			it follows from \eqref{derivata1} and \eqref{identita1} that
			\begin{equation*}
				\begin{split}
					A_{\overline{f}}[\varphi](x)&=D_x(\overline{f}(x,\varphi(x),D\varphi(x)))\cdot D_\xi\overline{f}(x,s,D\varphi)^T\\
					&=D_x(f(x,\varphi(x),X\varphi(x)))\cdot(D_pf(x,\varphi(x),X
					\varphi(x))\cdot C(x))^T\\
					&=D_x(f(x,\varphi(x),X\varphi(x)))\cdot C(x)^T\cdot D_pf(x,\varphi(x),X\varphi(x))^T\\
					&=X(f(x,\varphi(x),X\varphi(x)))\cdot D_pf(x,\varphi(x),X\varphi(x))^T=A_f[\varphi](x),
				\end{split}
			\end{equation*}
			whence $A_{\overline{f}} [\varphi](0)>0$. The claim then follows as in \cite[Theorem 1]{crandall} and thanks to \eqref{identita1}.
		\end{proof}
	\textbf{Step 3.}
		Now we want to exploit $\Psi_\epsilon$ as a test function in the definition of absolute minimizer on a suitable neighbourhood of $0$. For doing this let us notice that, thanks to \eqref{phiepsi},
		\begin{equation*}
			\begin{split}
				\Psi_\epsilon(x)&=\Psi_\epsilon(0)+D\Psi_\epsilon(0)\cdot x+x^T\cdot D^2\Psi_\epsilon(0)\cdot x + o(|x|^2)\\
				&=\phi(0)-\epsilon+D\phi(0)\cdot x+x^T\cdot D^2\Psi_\epsilon(0)\cdot x + o(|x|^2)\\
				&> \phi(0)-\epsilon+D\phi(0)\cdot x+x^T\cdot D^2\phi(0)\cdot x +2\mu|x|^2+ o(|x|^2)\\
				&=\phi(x)-\epsilon+2\mu|x|^2+ o(|x|^2)
			\end{split}
		\end{equation*}
		as $x$ goes to zero. Therefore we have that
		\begin{equation}\label{stimataylor}
			\Psi_\epsilon(x)>\phi(x)-\epsilon+\mu|x|^2
		\end{equation}
		for any $x\in \overline{B_{R_3}(0)}\setminus\{0\}$, for any $\epsilon\in[0,\epsilon_1]$ and for some $R_3<R_2$ sufficiently small. Let now $0<\epsilon_2<\epsilon_1$ small enough such that $\sqrt{\frac{\epsilon}{\mu}}<R_3$ for any $\epsilon\in[0,\epsilon_2]$ and define $\mathcal{N}_\epsilon$ as the connected component of
		$$\{x\in B_{R_3}(0)\,:\,\Psi_\epsilon(x)<u(x)\}$$
		containing zero (note that $\Psi_\epsilon(0)=u(0)-\epsilon<u(0)$ if $\epsilon>0$). Therefore $\mathcal{N}_\epsilon$ is an open and connected neighborhood of $0$ for any $\epsilon\in(0,\epsilon_2]$. Moreover, since \eqref{stimataylor} implies that
		\begin{equation*}
			\Psi_\epsilon(x)>\phi(x)\geq u(x)\qquad\text{on }\partial B_{\sqrt{\frac{\epsilon}{\mu}}}(0),
		\end{equation*}
		it follows that 
		\begin{equation}\label{contradue}
			\mathcal{N}_\epsilon\subseteq B_{\sqrt{\frac{\epsilon}{\mu}}}(0)\subsetneqq B_{R_3}(0),
		\end{equation}
		which implies that
		\begin{equation*}
			u|_{\partial\mathcal{N}_\epsilon}=\Psi_\epsilon|_{\partial\mathcal{N}_\epsilon}.
		\end{equation*}
		Being $u$ an absolute minimizer, and recalling \eqref{phiepsi}, we conclude that
		\begin{equation}\label{disug1}
			\begin{split}
				f(x,u(x),Xu(x))&\leq F(u,\mathcal{N}_\epsilon)\leq F(\Psi_\epsilon,\mathcal{N}_\epsilon)=f(0,\phi(0)-\epsilon,X\phi(0))=f(x,\Psi_\epsilon(x),X\Psi_\epsilon(x))
			\end{split}
		\end{equation}
		for a.e. $x\in\mathcal{N}_\epsilon$ and for any $\epsilon\in[0,\epsilon_2]$.\\
		
		\textbf{Step 4.} Got to this point we wish to achive the situation in which $s\mapsto f(x,s,p)$ is non-decreasing locally in a 
		neighborhood of $(0,\phi(0),X\phi(0))$. Therefore we follow the strategy of \cite{crandall} and we show that, via a suitable change of variables, this assumption is possible. Let us define then a new function $g$ as 
		\begin{equation*}
			g(x,s,p):=f(x,u(0)+q\cdot x+G(s),q\cdot C(0)^T+G'(s)p)
		\end{equation*}
		for any $(x,s,p)$ in a suitable neighborhood of $(0,\phi(0),X\phi(0))$, where $q\in\Rn$ has to be determined and $G\in C^\infty(-\delta,\delta)$ is a local increasing diffeomorphism such that $G(0)=0$ and $G'(0)>0$. Let us notice that $g$  is $C^2$ and quasiconvex in the third argument.
		Moreover, if we define $\overline{u}$ and $\overline{\phi}$ in a neighborhood of $0$ by requiring that
		\begin{equation*}
			u(x)=u(0)+q\cdot x+G(\overline{u}(x)),
		\end{equation*}
		\begin{equation}\label{phibarra}
			\phi(x)=\phi(0)+q\cdot x+G(\overline{\phi}(x)),
		\end{equation}
		it is easy to see that \eqref{maximumvisco} holds for $\overline{u}$ and $\overline{\phi}$ and that $\overline{\phi}(0)=\overline{u}(0)=0$.
		If $H$ is the supremal functional associated to $g$ it is easy to see that $\overline{u}$ is an absolute minimizer for $H$ (we stress that we are working in a suitable neighborhood of $0$).
		Easy computations show that
		\begin{equation*}
			D_x g=D_xf+D_sf q,\qquad D_sg=G'(s)D_sf+G''(s)D_pf\cdot p^T,\qquad D_pg=G'(s)D_pf.
		\end{equation*}
		Therefore, noticing that 
		\begin{equation*}
			g(x,\overline{\phi}(x),X\overline{\phi}(x))=f(x,\phi(x),X\phi(x))
		\end{equation*}
		for any $x$ in the usual neighborhood of $0$, we have that
		\begin{equation*}
			\begin{split}
				A_g[\overline\phi](x)&=-X(g(x,\overline{\phi}(x),X\overline{\phi}(x)))\cdot D_pg(x,\overline{\phi}(x),X\overline{\phi}(x))\\
				&=-X(f(x,\phi(x),X\phi(x)))\cdot D_pg(x,\overline{\phi}(x),X\overline{\phi}(x))\\
				&=-X(f(x,\phi(x),X\phi(x)))\cdot(G'(\overline{\phi}(x))D_pf(x,\phi(x),X\phi(x))=G'(\overline{\phi}(x))A_f[\phi](x),
			\end{split}
		\end{equation*}
		and so $A_g[\overline{\phi}](0)=G'(0)A_f[\phi](0)>0$.
		Moreover, \eqref{phibarra} implies that
		\begin{equation*}
			X\overline{\phi}(0)=\frac{X\phi(0)-q\cdot C(0)^T}{G'(0)}.
		\end{equation*}
		Therefore we have that
		\begin{equation*}
			D_s g(0,\overline{\phi}(0),X\overline{\phi}(0))=G'(0)D_sf(0,\phi(0),X\phi(0))+\frac{G''(0)}{G'(0)}(X\phi(0)-q\cdot C(0)^T)\cdot D_pf(0,\phi(0),X\phi(0))^T.
		\end{equation*}
		Hence, if we choose $G$ as $G(s)=s+\frac{\beta}{2}s^2$, where $\beta>0$, and we choose $q$ as
		\begin{equation*}
			q:=D\phi(0)+D_xf(0,\phi(0),X\phi(0))+D_sf(0,\phi(0),X\phi(0))D\phi(0)+D_pf(0,\phi(0),X\phi(0))\cdot B,
		\end{equation*}
		where $B$ is the $m\times n$ matrix defined as 
		\begin{equation*}
			B_{ij}:=\frac{\partial}{\partial x_j}X_i\phi(x)\bigg|_{x=0}
		\end{equation*}
		for any $i=1,\ldots,m$ and $j=1,\ldots,n$, and noticing that
		\begin{equation*}
			p\cdot B\cdot C(0)^T\cdot p^T=p\cdot X^2\phi(0)\cdot p^T
		\end{equation*}
		for any $p\in\Rm$, thanks to \eqref{contradiction} we conclude that
		\begin{equation*}
			D_s g(0,\overline{\phi}(0),X\overline{\phi}(0))=D_sf(0,\phi(0),X\phi(0))+\beta A_f[\phi](0)>0,
		\end{equation*}
		provided we choose $\beta$ sufficiently big. Therefore, up to work in this new setting, we can assume that $s\mapsto f(x,s,p)$ is increasing in a neighborhood of $(0,\phi(0),X\phi(0))$. This fact and \eqref{disug1} allow to find $0<\epsilon_3<\epsilon_2$ such that
		\begin{equation}\label{disug2}
			f(x,u(x),Xu(x))\leq f(x,u(x),X\Psi_\epsilon(x))
		\end{equation}
		for any $\epsilon\in(0,\epsilon_3]$ and for a.e. $x\in\mathcal{N}_\epsilon$.\\
		\textbf{Step 5.}
		We are going to exploit \eqref{disug2}, together with Proposition \ref{quasiconvex}, in a suitable way. For doing this let us consider the first-order system of ODEs
		
		\begin{equation}\label{ode}
			\begin{cases}
				\Dot{\gamma}(t)=-C(\gamma(t))^T\cdot D_pf(\gamma(t),u(\gamma(t),X\Psi_\epsilon(\gamma(t)))^T \\ \gamma(0)=0
			\end{cases}
		\end{equation}
		and, for any $\epsilon\in[0,\epsilon_3]$ and a suitable $R_4<R_3$, we define $g_\epsilon:B_{R_4}(0)\scu\Rn$ as
		\begin{equation*}
			g_\epsilon(x):=-C(x)^T\cdot D_pf(x,u(x),X\Psi_\epsilon(x))^T.
		\end{equation*}
		It is easy to see (recall \eqref{continuita}) that $g_\epsilon\in C(B_{R_4}(0),\Rn)$. If we define $$\mathcal{C}:=\max_{i,j}\{\sup_{B_{R_4}(0)}|c_{ij}|\},$$
		it follows from our assumptions that $0<\mathcal{C}<+\infty$. Moreover, thanks to \eqref{continuita} and \eqref{peano1}, there exist $0<\epsilon_4<\epsilon_3$ and $0<R_5<R_4$ such that 
		\begin{equation*}
			\begin{aligned}
				|D\Psi_\epsilon(x)-D\phi(0)|\leq 1\\
				|u(x)-u(0)|\leq 1
			\end{aligned}
		\end{equation*}
		for any $x\in \overline{B_{R_5}(0)}$ and $\epsilon\in[0,\epsilon_4]$. 
		Therefore, if we let $M_\epsilon:=\max\{g_\epsilon(x)\,:\,x\in \overline{B_{R_5}(0)}\}$, it follows that
		\begin{equation*}
			\begin{split}
				\Vert g_\epsilon(x)\Vert_{L^\infty({B_{R_5}(0))}}&\leq\mathcal{C}
				\|D_pf(x,u(x),X\Psi_\epsilon(x))\|_{L^\infty(B_{R_5}(0))}\\
				&\leq\mathcal{C}\Vert D_pf(x,s,p)\Vert_{L^\infty (B_{R_5}(0)\times B_1(u(0))\times B_{\mathcal{C}}(D\phi(0))}:=M
			\end{split}
		\end{equation*}
		for any $\epsilon\in[0,\epsilon_4]$. Since \eqref{contradiction} implies that $M_\epsilon>0$, we conclude that $0<M_\epsilon<M$ for any $\epsilon\in[0,\epsilon_4]$.
		Therefore, if we let 
		$$\epsilon_5:=\min\left\{\epsilon_4,\frac{R_5}{M}\right\},$$
		Peano's Theorem (cf. e.g. \cite[Theorem 2.19]{teschl})
		guarantees the existence, for any $\epsilon\in[0,\epsilon_5]$, of a curve $\gamma_\epsilon\in C^1((-\epsilon_5,\epsilon_5),\Rn)$ which solves \eqref{ode}. Moreover, from the first line of \eqref{ode} it follows that $\gamma_\epsilon$ is an horizontal curve.
		Then, Propositions \ref{quasiconvex} and \ref{derivabilitaac}, together with Lemma \ref{sublevel} and \eqref{disug2}, imply that \begin{equation*}
			\frac{d}{dt}\left(\Psi_\epsilon(\gamma_\epsilon(t))-u(\gamma_\epsilon(t)\right)\Bigg|_{t=t_0}=D_pf(\gamma_\epsilon(t_0),u(\gamma_\epsilon(t_0)),X\Psi_\epsilon(\gamma(t_0)))\cdot(g(t_0)-X\Psi_\epsilon(\gamma(t_0)))\leq 0
		\end{equation*}
		for a.e. $t_0\in(-\epsilon_5,\epsilon_5)$ and for any $\epsilon\in[0,\epsilon_5)$, and where $g(t_0)$ is as in Proposition \ref{derivabilitaac}. Therefore, if we fix $t_0\in(0,\epsilon_5)$, the previous inequality implies that
		\begin{eqnarray*}
			\Psi_\epsilon(\gamma_\epsilon(t_0))&=&\Psi_\epsilon(0)+\int_0^{t_0}\frac{d\Psi_\epsilon(\gamma_\epsilon(t))}{dt}dt\\
			&\leq&u(0)-\epsilon+\int_0^{t_0}\frac{du(\gamma_\epsilon(t))}{dt}dt\\
			&=&u(\gamma_\epsilon(t_0))-\epsilon<u(\gamma_\epsilon(t_0)),
		\end{eqnarray*}
		hence we conclude that $\gamma_\epsilon(t_0)\in\mathcal{N}_\epsilon$, which implies, together with \eqref{contradue}, that
		\begin{equation}\label{contrafine}
			\gamma_\epsilon(t_0)\in B_{\sqrt{\frac{\epsilon}{\mu}}}(0)
		\end{equation}
		for any $t_0\in[0,\epsilon_5)$ and any $\epsilon\in(0,\epsilon_5)$.
		On the other hand, the classical Taylor's formula applied to $\gamma_\epsilon$ implies that
		\begin{equation}\label{taylorgamma}
			\gamma_\epsilon(t)=-C(0)^T\cdot(D_pf(0,\phi(0),X\phi(0))^T t+o(t)
		\end{equation}
		as $t$ tends to zero and for any $\epsilon\in(0,\epsilon_5)$. If we let $2K:=|C(0)^T\cdot(D_pf(0,\phi(0),X\phi(0))^T|$, \eqref{contradiction} says that $2K>0$. Therefore, thanks to \eqref{taylorgamma}, we know that there exists $0<\epsilon_6<\epsilon_5$ such that
		\begin{equation}\label{stimaa}
			|\gamma_\epsilon(t)|\geq Kt
		\end{equation}
		for any for any $t,\epsilon\in(0,\epsilon_6)$.
		Let us choose $\overline{\epsilon}\in(0,\epsilon_6)$ such that 
		$$t_0:=\frac{2}{K}\sqrt{\frac{\overline{\epsilon}}{\mu}}<\epsilon_6.$$
		Then \eqref{stimaa} yelds that $|\gamma_{\overline{\epsilon}}(t_0)|\geq2\sqrt{\frac{\overline{\epsilon}}{\mu}}$, which is a clear contradiction with \eqref{contrafine}.
		
	\end{proof}

	\section{Appendix}
	\begin{proof}[Proof of Lemma \ref{lemmauno}]
		Let $z\in \cl A_k$ for any $k\geq 1$. Then for any $k\geq 1$ there exists a sequence $(z^k_h)_h\subseteq A_k$ converging to $z$ as $h$ goes to infinity.
		Therefore we can select a subsequence $(z^k)^k\subseteq(z^k_h)_h^k$ which converges to $z$ as $k$ goes to infinity and such that $z^k\in A_k$ for any $k\geq 1$.
		SInce $z^k\in A_k$, then there exist $y^k\in B_{1/k}(x)\setminus N$ such that $Xu(y^k)=z^k$. It follows that $y^k$ converges to $x$ as $k$ goes to infinity, $y^k\notin N$ and
		\begin{equation*}
			z=\lim_{k\to\infty}z^k=\lim_{k\to\infty}Xu(y^k).
		\end{equation*}
		We conclude that $z\in S$.
	\end{proof}
	
	\begin{proof}[Proof of Lemma \ref{lemmadue}]
		Let $z\in \ch(\cl A_k)$ for any $k\geq 1$. Then for any $k\geq 1$ there exists a sequence $(z^k_h)_h\subseteq co(\cl A_k)$ converging to $z$ as $h$ goes to infinity.
		As in the previous proof, let $(z^k)^k\subseteq(z^k_h)_h^k$ be a sequence which converges to $z$ as $k$ goes to infinity and such that $z^k\in co(\cl A_k)$ for any $k\geq 1$.
		Therefore, for each $k\geq 1$, there exist $(\lambda^{k}_1\ldots,\lambda^{k}_{m+1})\in\Lambda_{m+1}$ and $y^{k}_1,\ldots,y^{k}_{m+1}$ belonging to $\cl A_k$ such that
		$$z^k=\sum_{j=1}^{m+1}\lambda^{k}_jy^{k}_j.$$
		Up to passing to subsequences, we assume that 
		$$\lambda^{k}_j\to\lambda_j\qquad\text{ as }k\to\infty$$
		and
		$$y^{k}_j\to y_j\qquad\text{ as }k\to\infty$$
		for any $j=1,\ldots,m+1$. It is easy to see that $(\lambda_1,\ldots,\lambda_{m+1})\in\Lambda_{m+1}$ and that $y_j$ belongs to $\cl A_k$ for any $k\geq 1$. Therefore, thanks to our hypotheses, we have that $y^h_j\in S$. If we set 
		$$x:=\sum_{j=1}^{m+1}\lambda_jy_j,$$
		then $x\in co( S)$. Moreover, it holds that
		\begin{equation*}
			x=\sum_{j=1}^{m+1}\lambda_jy_j=\sum_{j=1}^{m+1}\lim_{k\to\infty}\lambda_j^{k}y_j^{k}=\lim_{k\to\infty}\sum_{j=1}^{m+1}\lambda_j^{k}y_j^{k}=\lim_{k\to\infty}z^k=z,
		\end{equation*}
		which implies that $z\in co(S)$.
	\end{proof}

\end{document}